\documentclass[11pt]{article}

\usepackage[margin=1in]{geometry}
\usepackage{amsmath,amssymb,amsthm,mathtools}
\usepackage{microtype}
\usepackage{enumitem}
\usepackage{xspace}
\usepackage{aliascnt}
\usepackage{hyperref}
\usepackage[dvipsnames]{xcolor}
\usepackage[capitalize,nameinlink]{cleveref}
\hypersetup{
    colorlinks=true,
    pdfpagemode=UseNone,
    citecolor=OliveGreen,
    linkcolor=NavyBlue,
    urlcolor=Magenta,
    pdfstartview=FitW
}

\theoremstyle{plain}
\newtheorem{theorem}{Theorem}[section]
\newaliascnt{lemma}{theorem}
\newtheorem{lemma}[lemma]{Lemma}
\aliascntresetthe{lemma}
\newaliascnt{proposition}{theorem}
\newtheorem{proposition}[proposition]{Proposition}
\aliascntresetthe{proposition}
\newaliascnt{corollary}{theorem}
\newtheorem{corollary}[corollary]{Corollary}
\aliascntresetthe{corollary}

\theoremstyle{definition}
\newaliascnt{definition}{theorem}
\newtheorem{definition}[definition]{Definition}
\aliascntresetthe{definition}
\newaliascnt{example}{theorem}
\newtheorem{example}[example]{Example}
\aliascntresetthe{example}

\theoremstyle{remark}
\newaliascnt{remark}{theorem}
\newtheorem{remark}[remark]{Remark}
\aliascntresetthe{remark}
\newaliascnt{note}{theorem}

\aliascntresetthe{note}

\crefname{theorem}{theorem}{theorems}
\Crefname{theorem}{Theorem}{Theorems}
\crefname{lemma}{lemma}{lemmas}
\Crefname{lemma}{Lemma}{Lemmas}
\crefname{proposition}{proposition}{propositions}
\Crefname{proposition}{Proposition}{Propositions}
\crefname{corollary}{corollary}{corollaries}
\Crefname{corollary}{Corollary}{Corollaries}
\crefname{definition}{definition}{definitions}
\Crefname{definition}{Definition}{Definitions}
\crefname{example}{example}{examples}
\Crefname{example}{Example}{Examples}
\crefname{exercise}{exercise}{exercises}
\Crefname{exercise}{Exercise}{Exercises}
\crefname{remark}{remark}{remarks}
\Crefname{remark}{Remark}{Remarks}
\crefname{note}{note}{notes}
\Crefname{note}{Note}{Notes}

\def\*#1{\mathbf{#1}}
\def\+#1{\mathcal{#1}}
\def\-#1{\mathrm{#1}}
\newcommand{\bb}{\mathbb}

\newcommand{\tp}[1]{\left(#1\right)}
\newcommand{\stp}[1]{\left[#1\right]}
\newcommand{\set}[1]{\left\{#1\right\}}
\newcommand{\abs}[1]{\left|#1\right|}
\newcommand{\norm}[1]{\left\|#1\right\|}
\newcommand{\ceil}[1]{\left\lceil#1\right\rceil}

\newcommand{\inner}[2]{\left\langle#1,\,#2\right\rangle}

\newcommand{\cmid}{\,:\,}

\newcommand{\cov}[1]{\mathsf{cov}\tp{#1}}

\newcommand{\dd}{\,\-d}
\newcommand{\eps}{\varepsilon}
\newcommand{\defeq}{:=}

\DeclareMathOperator{\supp}{supp}

\def\Emery{\'Emery\xspace}

\def\Poincare{Poincar\'e\xspace}

\title{Optimal Mixing of Glauber Dynamics for the Sherrington--Kirkpatrick Model at \(\beta < 1 / 2\)}
\author{Sihan Wang\thanks{Zhiyuan College, Shanghai Jiao Tong University, Shanghai, China. Email: \texttt{wangsihan\_leo@sjtu.edu.cn}.}}
\date{August 28, 2026}

\begin{document}
\pagenumbering{gobble}

\maketitle

\begin{abstract}
    We prove that for every fixed inverse temperature \(\beta < 1 / 2\), with high probability over the disorder, the single-site Glauber dynamics for the \(n\)-spin Sherrington--Kirkpatrick model mixes from every initial configuration to within total variation distance \(\eps\) in \(O_{\beta}\tp{n \log\tp{n / \eps}}\) steps. The bound holds uniformly over all external fields and is optimal up to constants depending only on \(\beta\). The main ingredient is a deterministic criterion for optimal-order \Poincare inequalities in general Ising models, established via the integrated Bakry--\Emery criterion together with a new two-spin estimate. A standard application of the localization-scheme framework of Chen and Eldan then upgrades the \Poincare inequality to a modified log-Sobolev inequality, yielding the optimal mixing-time bound. The main ideas underlying the proof of the \Poincare inequality were generated by GPT-5.6 Sol Ultra.
\end{abstract}

\tableofcontents

\newpage
\pagenumbering{arabic}

\section{Introduction}
\label{sec:introduction}

The Sherrington--Kirkpatrick (SK) model, introduced in \cite{SK75}, is a canonical mean-field model of spin glasses in statistical physics. A central algorithmic question is whether its Gibbs measure can be approximately sampled by Glauber dynamics in polynomial time up to \(\beta = 1\) \cite[Open Problem~15]{BKMR25}. Physical theory predicts fast convergence to equilibrium, at least for suitable observables and initializations, throughout the high-temperature regime \(\beta < 1\) \cite{SZ81,MPV87}.

We begin by introducing the SK model and Glauber dynamics. Let \(\*J \in \bb{R}^{n \times n}\) and \(\*h \in \bb{R}^n\). The Ising model with interaction \(\*J\) and external field \(\*h\) is the probability measure on \(\set{\pm 1}^n\) given by
\begin{equation}
    \mu_{\*J, \*h}\tp{\*x} \propto \exp\tp{\frac{1}{2} \*x^\top \*J \*x + \*h^\top \*x}, \quad \*x \in \set{\pm 1}^n.
    \label{eq:ising-law}
\end{equation}
The \(n\)-dimensional SK model is the Ising model whose symmetric off-diagonal interaction entries are independent and satisfy
\begin{equation}
    J_{ij} \sim \+N\tp{0, \frac{\beta^2}{n}}, \qquad 1 \le i < j \le n,
    \label{eq:sk-matrix}
\end{equation}
where \(\+N\tp{\mu, \sigma^2}\) denotes the Gaussian law with mean \(\mu\) and variance \(\sigma^2\). Equivalently, one may start from \(\*J = \beta \*W\) with \(\*W \sim \-{GOE}\tp{n}\) and discard its diagonal. Since the diagonal entries do not affect the Ising law, we take \(J_{ii} = 0\) for all \(i \in \stp{n}\). The parameter \(\beta > 0\) is the inverse temperature.

A standard Markov chain Monte Carlo (MCMC) method for sampling from spin systems is Glauber dynamics. Let \(\mu\) be a probability measure on \(\set{\pm 1}^n\). At each step, the Glauber dynamics targeting \(\mu\) updates the current configuration \(\*x \in \set{\pm 1}^n\) as follows:
\begin{itemize}
    \item Select a coordinate \(i \sim \-{Unif}\tp{\stp{n}}\).
    \item Obtain the new configuration \(\*x'\) by flipping the \(i\)-th coordinate of \(\*x\) with probability \(\frac{\mu\tp{\*x^{\oplus i}}}{\mu\tp{\*x} + \mu\tp{\*x^{\oplus i}}}\), where \(\*x^{\oplus i}\) denotes \(\*x\) with its \(i\)-th coordinate flipped.
\end{itemize}
Throughout this paper, we denote by \(P_{\*J, \*h}\) the Glauber dynamics kernel targeting the Ising model \(\mu_{\*J, \*h}\). Its \(\eps\)-mixing time is defined as
\[
T_{\-{mix}}\tp{\eps; P_{\*J, \*h}} \defeq \sup_{\*x \in \set{\pm 1}^n} \min\set{t \in \bb{N} \cmid d_{\-{TV}}\tp{P_{\*J, \*h}^t\tp{\*x, \cdot}, \mu_{\*J, \*h}} \le \eps}.
\]
We say that the chain mixes rapidly if its mixing time is polynomial in \(n\).

We next recall replica symmetry in the zero-field SK model (i.e., \(\*h = \*0\)). Let \(\*x, \*y \sim \mu_{\*J, \*0}\) be two independent samples, called replicas, and define their overlap by
\[
R\tp{\*x, \*y} \defeq \frac{1}{n} \sum_{i = 1}^n x_i y_i.
\]
For \(\beta < 1\), the overlap concentrates around \(0\) as \(n \to \infty\): two independent replicas exhibit no nontrivial macroscopic overlap \cite{GT02}. This high-temperature phase is the replica-symmetric regime. The critical point is \(\beta = 1\): the model undergoes its spin-glass phase transition there, while replica symmetry is broken for \(\beta > 1\) \cite{Ton02}.

Replica symmetry provides a natural physical heuristic for rapid mixing. Informally, when \(\beta < 1\), one does not expect the Gibbs measure to decompose into large, well-separated regions of substantial mass that create bottlenecks for local dynamics. Correspondingly, there should be no large free-energy barriers that trap a local Markov chain for very long times, suggesting that Glauber dynamics can efficiently explore the equilibrium measure. Replica symmetry, however, is a thermodynamic property and does not by itself imply rapid mixing. Establishing this algorithmic prediction rigorously is therefore a distinct and challenging problem.

Before the present work, the strongest rapid-mixing guarantee for Glauber dynamics that was uniform over arbitrary external fields reached \(\beta \lessapprox 0.295\) \cite{AKV24}. In the zero-field model, polynomial-time samplers with vanishing total-variation error are known for \(\beta < 1 / 2\) \cite{DLSS26a,DLSS26b}. Throughout the broader replica-symmetric regime \(\beta < 1\), efficient sampling is known in normalized Wasserstein distance \cite{Cel24}, but not in total variation. Our main result establishes optimal mixing of Glauber dynamics throughout the regime \(\beta < 1 / 2\).

\begin{theorem}
    Let \(\*J\) denote the SK interaction matrix at inverse temperature \(\beta < 1 / 2\), as defined in \eqref{eq:sk-matrix}. Then with probability at least \(1 - o_{\beta}\tp{1}\), the Glauber dynamics kernel \(P_{\*J, \*h}\) satisfies
    \[
    T_{\-{mix}}\tp{\eps; P_{\*J, \*h}} \le C_{\beta} n \log\tp{\frac{n}{\eps}}, \quad \forall \eps \in \tp{0, 1}
    \]
    simultaneously over all external fields \(\*h \in \bb{R}^n\), where \(C_{\beta} > 0\) is a constant depending only on \(\beta\).
    \label{thm:sk-optimal-mixing}
\end{theorem}

Thus, throughout the regime \(\beta < 1 / 2\), our result gives the strongest natural form of efficient sampling: convergence in total variation from the worst initial condition, uniformly over arbitrary external fields, using the canonical Glauber dynamics, with the optimal \(O_{\beta}\tp{n \log n}\) mixing time. This order is sharp for random-scan single-site dynamics, even for product measures.

\subsection{Technical Overview}
\label{subsec:technical-overview}

We prove rapid mixing of Glauber dynamics by establishing functional inequalities. Define the Dirichlet form of the Glauber dynamics \(P_{\*J, \*h}\) by
\[
\+E_{P_{\*J, \*h}}\tp{f, g} \defeq \inner{f}{\tp{I - P_{\*J, \*h}} g}_{L^2\tp{\mu_{\*J, \*h}}}.
\]
A \Poincare inequality with constant \(\gamma > 0\) is an inequality of the form
\[
\gamma \*{Var}_{\mu_{\*J, \*h}}\stp{f} \le \+E_{P_{\*J, \*h}}\tp{f, f}, \quad \forall f \in L^2\tp{\mu_{\*J, \*h}}.
\]
A modified log-Sobolev inequality with constant \(\rho > 0\) is an inequality of the form
\[
\rho \*{Ent}_{\mu_{\*J, \*h}}\stp{f} \le \+E_{P_{\*J, \*h}}\tp{\log f, f}, \quad \forall f \in L^2\tp{\mu_{\*J, \*h}} \cmid f \ge 0.
\]
Standard arguments show that polynomial-time mixing from the worst initial state is equivalent to a \Poincare inequality with constant \(\gamma = 1 / \-{poly}\tp{n}\), given \(\min_{x \in \supp\tp{\mu}} \mu\tp{x} = 1 / \exp\tp{\-{poly}\tp{n}}\). Moreover, a modified log-Sobolev inequality yields entropy contraction and is the key ingredient for establishing the optimal \(O_{\beta}\tp{n \log n}\) mixing time.

Our proof of \Cref{thm:sk-optimal-mixing} consists of two main steps. First, we establish a \Poincare inequality with the optimal-order constant \(\gamma = \Omega_{\beta}\tp{1 / n}\) for the SK Glauber dynamics. The argument combines the integrated Bakry--\Emery formulation of the \Poincare inequality, a new two-spin estimate, and a random-matrix estimate for the SK model. This constitutes the main technical contribution of the paper. Second, a standard application of the localization-scheme framework of \cite{CE25} upgrades the \Poincare inequality to a modified log-Sobolev inequality with the optimal-order constant \(\rho = \Omega_{\beta}\tp{1 / n}\). Together with an additional warm-start argument, this yields the optimal mixing-time bound.

\paragraph{Optimal-order \Poincare inequality.}
The seminal work of Bakry and \Emery \cite{BE85} introduced a powerful framework for establishing functional inequalities for Markov processes. Bakry--\Emery theory has become fundamental in the analysis of Markov diffusion operators, with broad applications including log-concave sampling \cite{Che26}. For Markov jump processes, however, the classical pointwise Bakry--\Emery criterion is often too restrictive to yield sharp functional inequalities. A more suitable tool in this setting is the integrated Bakry--\Emery criterion, which is equivalent to the \Poincare inequality for reversible Markov processes.

We briefly state an equivalent form of the integrated criterion for the Glauber dynamics \(P_{\*J, \*h}\), without introducing the carr\'e du champ operators. The \Poincare inequality for \(P_{\*J, \*h}\) with constant \(\gamma\) is equivalent to
\[
\norm{\tp{I - P_{\*J, \*h}} f}_{L^2\tp{\mu_{\*J, \*h}}}^2 \ge \gamma \inner{f}{\tp{I - P_{\*J, \*h}} f}_{L^2\tp{\mu_{\*J, \*h}}}, \quad \forall f \in L^2\tp{\mu_{\*J, \*h}}.
\]
This formulation is particularly well suited to our analysis. Expanding the left-hand side isolates the off-diagonal contributions, whose control constitutes the main technical challenge. Our key ingredient is a two-spin estimate. Conditioning on the remaining \(n - 2\) spins, the conditional law of any pair of spins is a two-spin Ising model with interaction \(J_{ij}\) and an arbitrary effective external field. We analyze this two-spin model using a Taylor-expansion-type estimate. This approach is particularly effective for the SK model, where the interactions are dense but individually weak. Summing the resulting estimates over all pairs yields the following deterministic criterion, which we prove in \Cref{subsec:integrated-bochner-estimate}.

\begin{theorem}[Deterministic \Poincare Criterion]
    Let \(\mu_{\*J, \*h}\) be an Ising model on \(\set{\pm 1}^n\), where \(\*J\) is symmetric and has zero diagonal. Write \(\*T \defeq \tp{\tanh J_{ij}}_{i, j = 1}^n\). There exist universal constants \(\delta \in \tp{0, 1}\) and \(C > 0\) such that the following holds: Define
    \[
    \gamma \defeq \frac{1}{n} \tp{1 - \max\set{\norm{\*T}_{\-{op}}, \, 2 \norm{\*T \circ \*T}_{\-{op}}} - C r_3\tp{\*T}}, \qquad r_3\tp{\*T} \defeq \max_{i \in \stp{n}} \sum_{j = 1}^n \abs{T_{ij}}^3,
    \]
    where \(\*T \circ \*T\) denotes the Hadamard product of \(\*T\) with itself. Suppose that \(\max_{i, j \in \stp{n}} \abs{T_{ij}} \le \delta\) and \(\gamma > 0\). Then the Glauber dynamics kernel \(P_{\*J, \*h}\) satisfies the \Poincare inequality with constant \(\gamma\):
    \[
    \gamma \*{Var}_{\mu_{\*J, \*h}}\stp{f} \le \+E_{P_{\*J, \*h}}\tp{f, f}, \quad \forall f \in L^2\tp{\mu_{\*J, \*h}}
    \]
    simultaneously over all external fields \(\*h \in \bb{R}^n\).
    \label{thm:deterministic-poincare-criterion}
\end{theorem}

The matrices \(\*T\) and \(\*T \circ \*T\) capture the first- and second-order contributions from the two-spin calculation, respectively, while \(r_3\tp{\*T}\) controls the higher-order remainder. For dense models with individually weak pairwise interactions, \(\norm{\*T}_{\-{op}}\) is typically the dominant term in our \Poincare criterion. In comparison, rapid mixing of Glauber dynamics is known under the spectral-width condition \(\lambda_{\max}\tp{\*J} - \lambda_{\min}\tp{\*J} < 1\) \cite{EKZ22,AKJVP22}. When \(\lambda_{\max}\tp{\*J}\) and \(\lambda_{\min}\tp{\*J}\) are approximately symmetric about \(0\), as in the SK model, our criterion effectively improves on this spectral-width condition by a factor of \(2\) for dense models with weak pairwise interactions. Here, \(\*J\) is assumed to have zero diagonal.

\begin{example}
    As an illustrative example of the criterion, consider the ferromagnetic Ising model on the complete bipartite graph \(K_{n, n}\). The interaction matrix at inverse temperature \(\beta > 0\) is
    \[
    \*J = \frac{\beta}{n}
    \begin{bmatrix}
        \*O_{n \times n} & \*1_{n \times n} \\
        \*1_{n \times n} & \*O_{n \times n}
    \end{bmatrix},
    \]
    where \(\*O_{n \times n}\) and \(\*1_{n \times n}\) denote the all-zero and all-one matrices, respectively. A direct computation shows that the spectrum of \(\*J\) is \(\set{0, \pm \beta}\). Hence, the spectral-width condition \(\lambda_{\max}\tp{\*J} - \lambda_{\min}\tp{\*J} < 1\) requires \(\beta < 1 / 2\). In contrast,
    \[
    \max_{i, j \in \stp{n}} \abs{T_{ij}} = \Theta\tp{\frac{\beta}{n}}, \qquad\! \norm{\*T}_{\-{op}} = \beta - \Theta\tp{\frac{\beta^3}{n^2}}, \qquad\! \norm{\*T \circ \*T}_{\-{op}} = \Theta\tp{\frac{\beta^2}{n}}, \qquad\! r_3\tp{\*T} = \Theta\tp{\frac{\beta^3}{n^2}}.
    \]
    Thus, \Cref{thm:deterministic-poincare-criterion} gives the improved condition \(\beta < 1\) for a \Poincare inequality, matching the zero-field phase transition point \(\beta_c = 1\) for this model.
\end{example}

For the SK model, standard random-matrix estimates verify the conditions of our criterion when \(\beta < 1 / 2\). Informally,
\[
\norm{\*T}_{\-{op}} \approx \norm{\*J}_{\-{op}} = 2 \beta + o\tp{1},
\]
where the estimate for \(\norm{\*J}_{\-{op}}\) is a standard consequence of random matrix theory. Similarly,
\[
\norm{\*T \circ \*T}_{\-{op}} \approx \norm{\*J \circ \*J}_{\-{op}} \approx \norm{\frac{\beta^2}{n} \*1_{n \times n}}_{\-{op}} = \beta^2.
\]
Moreover, \(\max_{i, j \in \stp{n}} \abs{T_{ij}}\) and \(r_3\tp{\*T}\) are negligible with high probability by standard Gaussian tail bounds. Therefore, the criterion yields the following \Poincare constant for the SK Glauber dynamics at \(\beta < 1 / 2\):
\[
\gamma = \frac{1 - 2 \beta - o\tp{1}}{n}.
\]

\begin{remark}
    After completing this work, we learned through personal communication that Heng Guo and Xinyuan Zhang independently obtained a stronger version of \Cref{thm:deterministic-poincare-criterion} (work in preparation). In particular, they show that if \(\lambda_{\max}\tp{\*T} < 1\), then the Glauber dynamics \(P_{\*J, \*h}\) satisfies the \Poincare inequality with constant
    \[
    \gamma = \frac{1 - \lambda_{\max}\tp{\*T}}{n}.
    \]
    As an immediate consequence, one obtains a spectral-gap bound for the SK Glauber dynamics for \(\beta < 1 / 2\). Their proof uses techniques entirely different from ours.
\end{remark}

\paragraph{Optimal-order modified log-Sobolev inequality.}
Once the \Poincare inequalities are established, upgrading them to a modified log-Sobolev inequality is a standard application of the localization-scheme framework of \cite{CE25}:
\[
\begin{aligned}
    \text{\Poincare inequality for every exponential tilt} &\Longrightarrow \text{covariance bound for every exponential tilt} \\
    &\Longrightarrow \text{approximate conservation of entropy} \\
    &\Longrightarrow \text{modified log-Sobolev inequality}.
\end{aligned}
\]
Finally, a coupon-collector phase of \(O\tp{n \log\tp{n / \eps}}\) steps produces, except with probability at most \(\eps / 2\), an \(\exp\tp{\-{poly}\tp{n}}\)-warm start for \(\mu_{\*J, \*h}\). The modified log-Sobolev inequality then gives mixing after another \(O_{\beta}\tp{n \log\tp{n / \eps}}\) steps, proving \Cref{thm:sk-optimal-mixing}.

\subsection{Related Work}
\label{subsec:related-work}

\paragraph{Sampling from the SK model.}
One line of work on sampling from the SK model establishes functional inequalities to control the mixing time of Glauber dynamics. \cite{BB19} introduced the spectral-width condition \(\lambda_{\max}\tp{\*J} - \lambda_{\min}\tp{\*J} < 1\), under which a log-Sobolev inequality holds, yielding the first rapid-relaxation result for the SK model at \(\beta < 1 / 4\).\footnote{Their result does not directly imply rapid mixing of Glauber dynamics because of a mismatch in the Dirichlet form. Nonetheless, the same argument applies to the Glauber Dirichlet form and yields a modified log-Sobolev inequality for Glauber dynamics; see \cite[Section~6.2.5]{BBD24}.} Under the same spectral-width condition, \cite{EKZ22} established a \Poincare inequality for Glauber dynamics, thereby obtaining rapid mixing for \(\beta < 1 / 4\). Using entropic independence, \cite{AKJVP22} strengthened this result to a modified log-Sobolev inequality, yielding the optimal \(O\tp{n \log n}\) mixing time. Subsequently, \cite{AKV24} developed a trickle-down analysis for localization schemes and proved optimal Glauber mixing up to \(\beta \approx 0.295\), the strongest such guarantee prior to the present work. More recently, \cite{DLSS26b} established weak \Poincare inequalities for the zero-field SK model in the regime \(\beta < 1 / 2\) and used them to analyze Glauber dynamics from a suitable warm start; these inequalities do not imply rapid mixing from worst-case initial states. In a complementary direction, \cite{BEAR26} proved rapid mixing of Glauber dynamics for every fixed \(\beta > 0\), provided the model is subject to a sufficiently strong homogeneous external field, i.e., \(\*h = h\*1\) with \(\abs{h} > h_0\tp{\beta}\).

A parallel line of work develops samplers based on algorithmic stochastic localization (ASL). ASL was introduced in \cite{EAMS22} by discretizing the stochastic localization process of \cite{Eld13}. For the zero-field SK model with \(\beta < 1 / 2\), \cite{EAMS22} constructed, for every fixed accuracy in normalized Wasserstein distance, an \(O\tp{n^2}\)-time sampling algorithm. \cite{Cel24} extended this guarantee throughout the replica-symmetric regime \(\beta < 1\). More recently, \cite{DLSS26a} obtained an efficient sampler with vanishing total-variation error for \(\beta < 1 / 2\) by combining ASL, potential Hessian ascent, and a rejection-sampling correction, with running time \(O\tp{\-{poly}\tp{n} \exp\tp{O\tp{1 / \eps}}}\).

On the hardness side, \cite{Sel25} showed that, for the zero-field SK model, Glauber dynamics exhibits exponentially slow mixing from the worst-case initialization for sufficiently large \(\beta\). Moreover, \cite{EAMS22} proved that no stable algorithm can approximately sample from the zero-field SK model in Wasserstein distance when \(\beta > 1\). This result, however, does not rule out MCMC methods such as Glauber dynamics.

More broadly, sampling algorithms have been studied for Ising and spherical mixed \(p\)-spin glasses. Rapid mixing in high-temperature regimes has been established via \Poincare and (modified) log-Sobolev inequalities \cite{ABXY24,AJKVP24,GJ19}, while ASL and simulated annealing yield efficient sampling guarantees in broader parameter regimes \cite{EAMS25,HMP24,HMRW25}.

\paragraph{Integrated Bakry--\Emery criterion.}
The integrated Bakry--\Emery criterion provides an equivalent formulation of the \Poincare inequality and is particularly useful for the analysis of jump processes. \cite{BCDPP06} developed a general framework for reversible jump processes based on this criterion and applied it to interacting particle systems to obtain spectral-gap estimates. \cite{KKO13} extended this framework to infinite volume and applied it to Glauber dynamics for point processes. More recently, \cite{GJMPPS26} adapted the argument of \cite{KKO13} to probability distributions on finite downward-closed set systems, yielding a short proof of a \Poincare criterion introduced in \cite{CCCYZ25}. Separately, \cite{Joh17} used a related Bakry--\Emery-type argument for one-dimensional birth-death chains and formulated the condition of \(c\)-log-concavity. In an upcoming work by the author, the integrated Bakry--\Emery criterion is used to establish \Poincare inequalities for birth-death chains on \(\bb{N}^d\) and derive discrete Brascamp--Lieb inequalities, while also recovering the \Poincare criterion of \cite{CCCYZ25,GJMPPS26}.

\subsection{Notation}
\label{subsec:notations}

We write \(\bb{N} \defeq \set{0, 1, 2, \ldots}\) and \(\stp{n} \defeq \set{1, \ldots, n}\). Bold symbols denote vectors or matrices, as determined by context. For \(S \subseteq \stp{n}\), \(\*x_{-S}\) is the restriction of \(\*x\) to the coordinates outside \(S\), and \(\*x_{-i} \defeq \*x_{-\set{i}}\); the same convention applies to random vectors. For \(\*x \in \set{\pm 1}^n\), we write \(\*x^{\oplus i}\) for \(\*x\) with its \(i\)-th coordinate flipped and \(\*x^{i \gets a}\) for \(\*x\) with that coordinate replaced by \(a\).

For a probability measure \(\mu\) on a finite set \(\Omega\), a function \(f \colon \Omega \to \bb{R}\), and an event \(A \subseteq \Omega\), we write \(\*E_{\mu}\stp{f}\), \(\*{Var}_{\mu}\stp{f}\), and \(\*{Pr}_{\mu}\stp{A}\) for expectation, variance, and probability under \(\mu\), respectively, and omit the subscript when the underlying law is clear. If another probability measure \(\nu\) on \(\Omega\) is absolutely continuous with respect to \(\mu\), we denote its Radon--Nikodym derivative by \(\frac{\dd \nu}{\dd \mu}\). Our total variation convention is
\[
d_{\-{TV}}\tp{\nu, \mu} \defeq \frac{1}{2} \sum_{x \in \Omega} \abs{\nu\tp{x} - \mu\tp{x}}.
\]
For functions \(f, g \colon \Omega \to \bb{R}\), we write \(\inner{f}{g}_{L^2\tp{\mu}} \defeq \*E_{\mu}\stp{f g}\) and
\[
\*{Ent}_{\mu}\stp{f} \defeq \*E_{\mu}\stp{f \log f} - \*E_{\mu}\stp{f} \log \*E_{\mu}\stp{f}, \qquad f \ge 0,
\]
with the convention \(0 \log 0 = 0\). For a probability measure \(\mu\), let \(\supp\tp{\mu}\) denote its support. If \(\mu\) is a measure on \(\bb{R}^n\), let \(\cov{\mu}\) denote its covariance matrix.

For a real matrix \(\*A\), \(\norm{\*A}_{\-{op}}\) and \(\norm{\*A}_{\-{F}}\) denote its Euclidean operator and Frobenius norms, while \(\lambda_{\max}\tp{\*A}\) and \(\lambda_{\min}\tp{\*A}\) denote its extreme eigenvalues. We write \(\*A \circ \*B\) for the Hadamard product. Under our normalization, \(\*W \sim \-{GOE}\tp{n}\) has independent upper-triangular entries with \(W_{ij} \sim \+N\tp{0, 1 / n}\) for \(i < j\) and \(W_{ii} \sim \+N\tp{0, 2 / n}\). All asymptotics are as \(n \to \infty\). A subscript \(\beta\) on \(O\) or \(o\) allows dependence on fixed \(\beta\), and probabilities concerning the SK model are taken over \(\*J\).

\section{Preliminaries}
\label{sec:preliminaries}

\subsection{Basics of Markov Chains}
\label{subsec:markov-chains-basics}

Let \(\Omega\) be a finite state space. A Markov chain is characterized by its transition kernel \(P \in \bb{R}^{\Omega \times \Omega}\). Throughout the paper, we identify the Markov chain with its transition kernel \(P\).

\begin{definition}
    A probability measure \(\mu\) on \(\Omega\) is stationary for \(P\) if \(\mu P = \mu\).
\end{definition}

\begin{definition}
    The Markov chain \(P\) is irreducible if, for all \(x, y \in \Omega\), there exists \(t \in \bb{N}\) such that \(P^t\tp{x, y} > 0\). It is aperiodic if, for every \(x \in \Omega\), \(\gcd\set{t \in \bb{N} \cmid P^t\tp{x, x} > 0} = 1\). It is ergodic if it is both irreducible and aperiodic.
\end{definition}

Ergodicity ensures that the stationary law is the unique long-time equilibrium.

\begin{theorem}[Fundamental Theorem of Markov Chains; see, e.g., \cite{LPW17}]
    Let \(P\) be an ergodic Markov chain on a finite state space \(\Omega\). Then \(P\) has a unique stationary distribution \(\mu\), and for every initial distribution \(\nu\) on \(\Omega\), \(\nu P^t \to \mu\) pointwise as \(t \to \infty\).
\end{theorem}

We measure the resulting convergence in total variation. Besides worst-case initial states, we use warm starts, whose densities relative to stationarity are uniformly bounded.

\begin{definition}
    Let \(P\) be an ergodic Markov chain on a finite state space \(\Omega\) with stationary distribution \(\mu\). For \(\eps \in \tp{0, 1}\), the \(\eps\)-mixing time of \(P\) is
    \[
    T_{\-{mix}}\tp{\eps; P} \defeq \sup_{x \in \Omega} \min\set{t \in \bb{N} \cmid d_{\-{TV}}\tp{P^t\tp{x, \cdot}, \mu} \le \eps}.
    \]
    For \(M \ge 1\), its \(M\)-warm \(\eps\)-mixing time is
    \[
    T_{\-{mix}}^{\tp{M}}\tp{\eps; P} \defeq \sup_{\frac{\dd \nu}{\dd \mu} \le M} \min\set{t \in \bb{N} \cmid d_{\-{TV}}\tp{\nu P^t, \mu} \le \eps},
    \]
    where the supremum is over probability measures \(\nu\) on \(\Omega\).
\end{definition}

\begin{remark}
    It is straightforward to show that
    \[
    T_{\-{mix}}\tp{\eps; P} \le T_{\-{mix}}^{\tp{1 / \mu_{\min}}}\tp{\eps; P}, \qquad \mu_{\min} \defeq \min_{x \in \supp\tp{\mu}} \mu\tp{x}.
    \]
\end{remark}

\subsection{Functional Inequalities and Mixing Time}
\label{subsec:functional-inequalities-mixing}

Functional inequalities quantify convergence by comparing fluctuations or entropy with the Dirichlet form. We use the following discrete-time normalization.

\begin{definition}
    Let \(P\) be a Markov chain on a finite state space \(\Omega\) with stationary distribution \(\mu\). The generator and Dirichlet form of \(P\) are defined by
    \[
    \+L \defeq P - I, \qquad \+E_P\tp{f, g} \defeq -\inner{f}{\+L g}_{L^2\tp{\mu}}.
    \]
\end{definition}

We use a \Poincare inequality to control variance contraction and a modified log-Sobolev inequality to control entropy contraction.

\begin{definition}
    Let \(P\) be an ergodic Markov chain on a finite state space \(\Omega\) with stationary distribution \(\mu\).
    \begin{itemize}
        \item We say that \(P\) satisfies a \Poincare inequality with constant \(\gamma > 0\) if
        \[
        \gamma \*{Var}_{\mu}\stp{f} \le \+E_P\tp{f, f}, \quad \forall f \in L^2\tp{\mu}.
        \]
        \item We say that \(P\) satisfies a modified log-Sobolev inequality with constant \(\rho > 0\) if
        \[
        \rho \*{Ent}_{\mu}\stp{f} \le \+E_P\tp{\log f, f}, \quad \forall f \in L^2\tp{\mu} \cmid f \ge 0.
        \]
    \end{itemize}
\end{definition}

The corresponding constants yield the warm-start mixing bounds.

\begin{theorem}[Functional Inequalities and Mixing Time; see, e.g., \cite{LPW17,BT03}]
    Let \(P\) be an ergodic Markov chain on a finite state space \(\Omega\) with stationary distribution \(\mu\), and suppose that \(P\) is positive semidefinite on \(L^2\tp{\mu}\). For \(M > 1\), the following statements hold:
    \begin{itemize}
        \item If \(P\) satisfies a \Poincare inequality with constant \(\gamma\), then
        \[
        T_{\-{mix}}^{\tp{M}}\tp{\eps; P} \le \frac{1}{\gamma} \tp{\log\frac{1}{2 \eps} + \frac{1}{2} \log M}, \quad \forall \eps \in \tp{0, 1}.
        \]
        \item If \(P\) satisfies a modified log-Sobolev inequality with constant \(\rho\), then
        \[
        T_{\-{mix}}^{\tp{M}}\tp{\eps; P} \le \frac{1}{\rho} \tp{\log\frac{1}{2 \eps^2} + \log\log M}, \quad \forall \eps \in \tp{0, 1}.
        \]
    \end{itemize}
    \label{thm:functional-inequality-mixing}
\end{theorem}

\subsection{Integrated Bakry--\Emery Criterion}
\label{subsec:integrated-bakry-emery}

We first recall the carr\'e du champ, iterated carr\'e du champ, and reversibility.

\begin{definition}
    Let \(P\) be a Markov chain on a finite state space \(\Omega\) with stationary distribution \(\mu\), and let \(\+L\) denote its generator. The carr\'e du champ operator \(\Gamma\) and the iterated carr\'e du champ operator \(\Gamma_2\) associated with \(P\) are defined by
    \[
    \Gamma\tp{f, g} \defeq \frac{1}{2} \tp{\+L\tp{fg} - f \+L g - g \+L f}, \qquad \Gamma_2\tp{f, g} \defeq \frac{1}{2} \tp{\+L\Gamma\tp{f, g} - \Gamma\tp{f, \+L g} - \Gamma\tp{g, \+L f}}.
    \]
\end{definition}

\begin{definition}
    Let \(P\) be a Markov chain on a finite state space \(\Omega\). We say that \(P\) is reversible with respect to a probability measure \(\mu\) on \(\Omega\) if the detailed balance condition holds:
    \[
    \mu\tp{x} P\tp{x, y} = \mu\tp{y} P\tp{y, x}, \quad \forall x, y \in \Omega.
    \]
    Equivalently, \(P\) is self-adjoint on \(L^2\tp{\mu}\).
\end{definition}

For a reversible chain, the integrated curvature bound below is equivalent to the \Poincare inequality.

\begin{theorem}[{Integrated Bakry--\Emery Criterion; cf. \cite[Proposition~4.8.3]{BGL14}}]
    Let \(P\) be an ergodic Markov chain on a finite state space \(\Omega\) that is reversible with respect to \(\mu\). Then \(P\) satisfies a \Poincare inequality with constant \(\gamma\) if and only if
    \[
    \*E_{\mu}\stp{\Gamma_2\tp{f, f}} \ge \gamma \*E_{\mu}\stp{\Gamma\tp{f, f}}, \quad \forall f \in L^2\tp{\mu}.
    \]
\end{theorem}

\section{Optimal-Order \Poincare Inequality}
\label{sec:optimal-poincare}

In this section, we establish an optimal-order \Poincare inequality for the SK model when \(\beta < 1 / 2\), simultaneously over arbitrary external fields. Our argument uses the integrated Bakry--\Emery criterion for Glauber dynamics, which is equivalent to the \Poincare inequality. We first derive the deterministic \Poincare criterion \Cref{thm:deterministic-poincare-criterion} in \Cref{subsec:integrated-bochner-estimate}, whose key input is the two-spin estimate proved in \Cref{subsec:two-spin-estimate}. We then verify the criterion for the SK model using the random-matrix estimates in \Cref{subsec:sk-estimates-poincare}. The vanishing error sequences in these estimates can be chosen independently of \(\beta\). The main result of this section is the following theorem.

\begin{theorem}
    For every fixed \(\beta < 1 / 2\), let \(\*J\) be the SK interaction matrix at inverse temperature \(\beta\), as defined in \eqref{eq:sk-matrix}. There exists a deterministic positive sequence \(\eta_n \to 0\), independent of \(\beta\), such that for all sufficiently large \(n \ge n_0\tp{\beta}\), with probability at least \(1 - o_{\beta}\tp{1}\), the Glauber dynamics kernel \(P_{\*J, \*h}\) satisfies the following \Poincare inequality simultaneously over all external fields \(\*h \in \bb{R}^n\):
    \[
    \frac{1 - 2 \beta - \eta_n}{n} \*{Var}_{\mu_{\*J, \*h}}\stp{f} \le \+E_{P_{\*J, \*h}}\tp{f, f}, \quad \forall f \in L^2\tp{\mu_{\*J, \*h}}.
    \]
    \label{thm:sk-optimal-poincare}
\end{theorem}

\subsection{Integrated Bakry--\Emery Estimate}
\label{subsec:integrated-bochner-estimate}

Fix \(\*J\) and \(\*h\) throughout this subsection, and abbreviate \(\mu \defeq \mu_{\*J, \*h}\) and \(P \defeq P_{\*J, \*h}\). We assume that \(\*X \sim \mu\) and that \(\*J\) has zero diagonal.

For each \(i \in \stp{n}\), define the single-site resampling operator \(P_i\) by
\begin{equation}
    P_i f\tp{\*x} \defeq \*E\stp{f\tp{\*X} \mid \*X_{-i} = \*x_{-i}}.
    \label{eq:single-site-resampling-operator}
\end{equation}
The Glauber dynamics kernel \(P\) is the average of these single-site resampling operators:
\begin{equation}
    P = \frac{1}{n} \sum_{i = 1}^n P_i.
    \label{eq:kernel-coordinate-decomposition}
\end{equation}
Let \(\+L \defeq P - I\) denote the generator associated with \(P\). Accordingly, \(\+L\) admits the coordinate decomposition
\begin{equation}
    -\+L = \frac{1}{n} \sum_{i = 1}^n D_i, \qquad D_i \defeq I - P_i.
    \label{eq:generator-coordinate-decomposition}
\end{equation}

\begin{proposition}
    Let \(\Gamma\) and \(\Gamma_2\) denote the carr\'e du champ and iterated carr\'e du champ operators associated with \(P\). Then, for every \(f \in L^2\tp{\mu}\),
    \[
    \*E_{\mu}\stp{\Gamma\tp{f, f}} = \frac{1}{n} \sum_{i = 1}^n \norm{D_i f}_{L^2\tp{\mu}}^2, \qquad \*E_{\mu}\stp{\Gamma_2\tp{f, f}} = \frac{1}{n^2} \sum_{i, j = 1}^n \inner{D_i f}{D_j f}_{L^2\tp{\mu}}.
    \]
    \label{prop:integrated-carre-du-champ}
\end{proposition}

\begin{proof}[Proof of \Cref{prop:integrated-carre-du-champ}]
    For each \(i \in \stp{n}\), the conditional-expectation operator \(P_i\) is the orthogonal projection in \(L^2\tp{\mu}\) onto the subspace of functions measurable with respect to \(\*X_{-i}\). Hence \(P_i\) is self-adjoint and \(P_i^2 = P_i\), so \(D_i = I - P_i\) is also self-adjoint and satisfies \(D_i^2 = D_i\). Since \(P\) is reversible with respect to \(\mu\), the definitions of \(\Gamma\) and \(\Gamma_2\) give
    \[
    \*E_{\mu}\stp{\Gamma\tp{f, f}} = -\inner{f}{\+L f}_{L^2\tp{\mu}}
    \]
    and
    \[
    \*E_{\mu}\stp{\Gamma_2\tp{f, f}} = -\*E_{\mu}\stp{\Gamma\tp{f, \+L f}} = \inner{f}{\+L^2 f}_{L^2\tp{\mu}} = \norm{\+L f}_{L^2\tp{\mu}}^2.
    \]
    Using \eqref{eq:generator-coordinate-decomposition} and the self-adjoint idempotence of each \(D_i\), we obtain
    \[
    \*E_{\mu}\stp{\Gamma\tp{f, f}} = \frac{1}{n} \sum_{i = 1}^n \inner{f}{D_i f}_{L^2\tp{\mu}} = \frac{1}{n} \sum_{i = 1}^n \norm{D_i f}_{L^2\tp{\mu}}^2,
    \]
    and
    \[
    \*E_{\mu}\stp{\Gamma_2\tp{f, f}} = \frac{1}{n^2} \norm{\sum_{i = 1}^n D_i f}_{L^2\tp{\mu}}^2 = \frac{1}{n^2} \sum_{i, j = 1}^n \inner{D_i f}{D_j f}_{L^2\tp{\mu}},
    \]
    as claimed.
\end{proof}

For each site \(i \in \stp{n}\), define the local mean and variance by
\begin{equation}
    m_i\tp{\*x} \defeq \*E\stp{X_i \mid \*X_{-i} = \*x_{-i}}, \qquad v_i\tp{\*x} \defeq \*Var\stp{X_i \mid \*X_{-i} = \*x_{-i}}.
    \label{eq:local-mean-variance}
\end{equation}
Define the local difference operator \(d_i\) by
\begin{equation}
    d_i f\tp{\*x} \defeq \frac{f\tp{\*x^{i \gets +1}} - f\tp{\*x^{i \gets -1}}}{2}.
    \label{eq:local-difference-operator}
\end{equation}
The following lemma expresses the coordinate projection \(D_i f\) and its \(L^2\tp{\mu}\)-norm in terms of the local difference \(d_i f\).

\begin{lemma}
    For every \(i \in \stp{n}\) and \(f \in L^2\tp{\mu}\), the following identities hold:
    \[
    D_i f\tp{\*x} = \tp{x_i - m_i\tp{\*x}} d_i f\tp{\*x}, \quad \forall \*x \in \set{\pm 1}^n,
    \]
    \[
    \*E_{\mu}\stp{v_i \tp{d_i f}^2} = \norm{D_i f}_{L^2\tp{\mu}}^2.
    \]
    \label{lem:local-variance-difference-identity}
\end{lemma}

\begin{proof}[Proof of \Cref{lem:local-variance-difference-identity}]
    Fix \(i \in \stp{n}\) and \(\*x \in \set{\pm 1}^n\). Since \(x_i \in \set{\pm 1}\),
    \[
    f\tp{\*x} = \frac{f\tp{\*x^{i \gets +1}} + f\tp{\*x^{i \gets -1}}}{2} + x_i d_i f\tp{\*x}.
    \]
    Both the first term on the right-hand side and \(d_i f\tp{\*x}\) depend only on \(\*x_{-i}\). Taking the conditional expectation in \eqref{eq:single-site-resampling-operator} therefore gives
    \[
    P_i f\tp{\*x} = \frac{f\tp{\*x^{i \gets +1}} + f\tp{\*x^{i \gets -1}}}{2} + m_i\tp{\*x} d_i f\tp{\*x}.
    \]
    Subtracting this identity from the preceding one proves
    \[
    D_i f\tp{\*x} = \tp{x_i - m_i\tp{\*x}} d_i f\tp{\*x}.
    \]
    Since \(d_i f\tp{\*X}\) is measurable with respect to \(\*X_{-i}\), we have
    \[
    \*E\stp{\tp{D_i f\tp{\*X}}^2 \mid \*X_{-i}} = \tp{d_i f}^2 \*E\stp{\tp{X_i - m_i\tp{\*X}}^2 \mid \*X_{-i}} = v_i \tp{d_i f}^2.
    \]
    Taking expectations yields the second desired identity.
\end{proof}

The key technical input is the following estimate for a two-spin Ising model. Its proof is deferred to \Cref{subsec:two-spin-estimate}.

\begin{lemma}
    Suppose that \(\mu = \mu_{\*J, \*h}\) is a two-spin Ising model on \(\set{\pm 1}^2\); that is,
    \[
    \mu\tp{\*x} \propto \exp\tp{J_{12} x_1 x_2 + h_1 x_1 + h_2 x_2}, \quad \*x \in \set{\pm 1}^2.
    \]
    There exist universal constants \(\delta \in \tp{0, 1}\) and \(C > 0\) such that, whenever \(\abs{\tanh J_{12}} \le \delta\), the following inequality holds for every \(f \colon \set{\pm 1}^2 \to \bb{R}\):
    \[
    \begin{aligned}
        \inner{D_1 f}{D_2 f}_{L^2\tp{\mu}} \ge &-\tp{\tanh J_{12}} \inner{v_1 \tp{d_1 f}}{v_2 \tp{d_2 f}}_{L^2\tp{\mu}} \\
        &- 2 \tp{\tanh J_{12}}^2 \inner{m_1 v_1 \tp{d_1 f}}{m_2 v_2 \tp{d_2 f}}_{L^2\tp{\mu}} \\
        &- C \abs{\tanh J_{12}}^3 \tp{\norm{D_1 f}_{L^2\tp{\mu}}^2 + \norm{D_2 f}_{L^2\tp{\mu}}^2}.
    \end{aligned}
    \]
    \label{lem:two-spin-estimate}
\end{lemma}

Combining the integrated Bakry--\Emery criterion with the two-spin estimate yields the deterministic criterion for a \Poincare inequality.

\begin{proof}[Proof of \Cref{thm:deterministic-poincare-criterion}]
    Fix \(f \in L^2\tp{\mu}\), and define the random vectors
    \[
    \*a \defeq \tp{v_i \tp{d_i f}}_{i = 1}^n,
    \qquad
    \*b \defeq \tp{m_i v_i \tp{d_i f}}_{i = 1}^n.
    \]
    Also, let \(\delta\) and \(C\) be the universal constants in \Cref{lem:two-spin-estimate}.

    For each \(i < j\), condition on \(\*X_{-\set{i, j}}\). The conditional law of \(\tp{X_i, X_j}\) is a two-spin Ising model with interaction \(J_{ij}\) and some external field. Under this conditioning, the two-spin resampling operators, local means and variances, and difference operators coincide with the restrictions of \(D_i, D_j, m_i, m_j, v_i, v_j, d_i, d_j\). Hence, applying \Cref{lem:two-spin-estimate} to this conditional law and then averaging over \(\*X_{-\set{i, j}}\) gives
    \[
    \inner{D_i f}{D_j f}_{L^2\tp{\mu}} \ge -T_{ij} \*E_{\mu}\stp{a_i a_j} - 2 T_{ij}^2 \*E_{\mu}\stp{b_i b_j} - C \abs{T_{ij}}^3 \tp{\norm{D_i f}_{L^2\tp{\mu}}^2 + \norm{D_j f}_{L^2\tp{\mu}}^2}.
    \]
    Summing this inequality over \(i < j\) and applying \Cref{prop:integrated-carre-du-champ} yields
    \[
    n^2 \*E_{\mu}\stp{\Gamma_2\tp{f, f}} \ge \sum_{i = 1}^n \norm{D_i f}_{L^2\tp{\mu}}^2 - \*E_{\mu}\stp{\*a^\top \*T \*a} - 2 \*E_{\mu}\stp{\*b^\top \tp{\*T \circ \*T} \*b} - 2 C r_3\tp{\*T} \sum_{i = 1}^n \norm{D_i f}_{L^2\tp{\mu}}^2.
    \]
    Set \(\alpha \defeq \max\set{\norm{\*T}_{\-{op}}, \, 2 \norm{\*T \circ \*T}_{\-{op}}}\). Since \(v_i = 1 - m_i^2\), pointwise in \(\*x\) we have
    \[
    a_i^2 + b_i^2 = v_i^2 \tp{1 + m_i^2} \tp{d_i f}^2 = v_i \tp{1 - m_i^4} \tp{d_i f}^2 \le v_i \tp{d_i f}^2.
    \]
    Therefore, by the operator-norm bound and \Cref{lem:local-variance-difference-identity},
    \[
    \begin{aligned}
        \*E_{\mu}\stp{\*a^\top \*T \*a}
        + 2 \*E_{\mu}\stp{\*b^\top \tp{\*T \circ \*T} \*b}
        &\le \norm{\*T}_{\-{op}} \*E_{\mu}\stp{\norm{\*a}_2^2}
        + 2 \norm{\*T \circ \*T}_{\-{op}} \*E_{\mu}\stp{\norm{\*b}_2^2} \\
        &\le \alpha \sum_{i = 1}^n \*E_{\mu}\stp{v_i \tp{d_i f}^2} \\
        &= \alpha \sum_{i = 1}^n \norm{D_i f}_{L^2\tp{\mu}}^2.
    \end{aligned}
    \]
    Combining the preceding bounds with \Cref{prop:integrated-carre-du-champ}, we obtain
    \[
    \*E_{\mu}\stp{\Gamma_2\tp{f, f}}
    \ge \frac{1 - \alpha - 2 C r_3\tp{\*T}}{n} \*E_{\mu}\stp{\Gamma\tp{f, f}}.
    \]
    Renaming \(2 C\) as \(C\), the integrated Bakry--\Emery criterion gives the asserted \Poincare inequality.
\end{proof}

\subsection{The Two-Spin Estimate}
\label{subsec:two-spin-estimate}

In this subsection, we prove the two-spin estimate stated in \Cref{lem:two-spin-estimate}. The argument can be viewed as a Taylor expansion around \(\tanh J_{12} = 0\).

\begin{proof}[Proof of \Cref{lem:two-spin-estimate}]
    Set
    \[
    a \defeq \tanh h_1, \qquad b \defeq \tanh h_2, \qquad t \defeq \tanh J_{12}, \qquad A \defeq 1 - a^2, \qquad B \defeq 1 - b^2.
    \]
    When \(z \in \set{\pm 1}\), the identity \(e^{h z} = \cosh\tp{h} \tp{1 + z \tanh h}\) holds. It follows that
    \[
    \mu\tp{\*x} \propto \exp\tp{J_{12} x_1 x_2 + h_1 x_1 + h_2 x_2} \propto \tp{1 + a x_1}\tp{1 + b x_2}\tp{1 + t x_1 x_2}.
    \]
    Hence,
    \[
    \mu\tp{x_1, x_2} = \frac{\tp{1 + a x_1}\tp{1 + b x_2}\tp{1 + t x_1 x_2}}{4 \tp{1 + a b t}}.
    \]
    In particular,
    \[
    m_1\tp{x_2} = \frac{a + t x_2}{1 + a t x_2}, \qquad m_2\tp{x_1} = \frac{b + t x_1}{1 + b t x_1}, \qquad v_1\tp{x_2} = \frac{A \tp{1 - t^2}}{\tp{1 + a t x_2}^2}, \qquad v_2\tp{x_1} = \frac{B \tp{1 - t^2}}{\tp{1 + b t x_1}^2}.
    \]

    Expand \(f\) in the centered basis of the product law with means \(a\) and \(b\):
    \[
    f\tp{\*x} = c + p \tp{x_1 - a} + q \tp{x_2 - b} + r \tp{x_1 - a}\tp{x_2 - b}.
    \]
    Then
    \[
    d_1 f\tp{x_2} = p + r \tp{x_2 - b}, \qquad d_2 f\tp{x_1} = q + r \tp{x_1 - a}.
    \]
    By \Cref{lem:local-variance-difference-identity},
    \[
    D_1 f\tp{\*x} = \tp{x_1 - m_1\tp{x_2}} d_1 f\tp{x_2}, \qquad D_2 f\tp{\*x} = \tp{x_2 - m_2\tp{x_1}} d_2 f\tp{x_1}.
    \]
    Define
    \[
    Q \defeq \inner{D_1 f}{D_2 f}_{L^2\tp{\mu}} + t \inner{v_1 \tp{d_1 f}}{v_2 \tp{d_2 f}}_{L^2\tp{\mu}} + 2 t^2 \inner{m_1 v_1 \tp{d_1 f}}{m_2 v_2 \tp{d_2 f}}_{L^2\tp{\mu}}.
    \]
    Substituting the preceding formulas and collecting the coefficients of \(p q, p r, q r, r^2\) gives
    \[
    Q = A B \tp{\phi_0 p q + \phi_1 p r + \phi_2 q r + \phi_{12} r^2},
    \]
    where each coefficient is a rational function of \(a, b, t\). Direct expansion at \(t = 0\) gives
    \[
    \begin{aligned}
        \phi_0 &= \tp{-6 a^2 b^2 + 2 a^2 + 2 b^2 - 1} t^3 + O\tp{t^4}, \\
        \phi_1 &= -b \tp{a^2 - 1}\tp{18 a^2 b^2 - 14 a^2 - 6 b^2 + 5} t^4 + O\tp{t^5}, \\
        \phi_2 &= -a \tp{b^2 - 1}\tp{18 a^2 b^2 - 6 a^2 - 14 b^2 + 5} t^4 + O\tp{t^5}, \\
        \phi_{12} &= 1 + O\tp{t^2}.
    \end{aligned}
    \]
    Indeed, the coefficients have denominators formed from \(1 + a b t\), \(1 \pm a t\), and \(1 \pm b t\), which are uniformly bounded away from zero whenever \(\abs{t} \le \delta_0 < 1\). Thus all Taylor remainders above are uniform. Taking a universal \(\delta > 0\) sufficiently small, we obtain
    \[
    \abs{\phi_0} \le C_0 \abs{t}^3, \qquad \abs{\phi_1} + \abs{\phi_2} \le C_0 \abs{t}^4, \qquad \phi_{12} \ge \frac{1}{2}.
    \]
    Hence,
    \[
    Q \ge A B \tp{-C_0 \abs{t}^3 \abs{p q} - C_0 \abs{t}^4 \tp{\abs{p r} + \abs{q r}} + \frac{1}{2} r^2} \ge - C_0 A B \abs{t}^3 \tp{\abs{p q} + \abs{p r} + \abs{q r}}.
    \]

    On the other hand, for the right-hand side of the inequality in \Cref{lem:two-spin-estimate}, a direct computation gives
    \[
    \begin{aligned}
        \norm{D_1 f}_{L^2\tp{\mu}}^2 &= \frac{A \tp{1 - t^2}}{1 + a b t} \sum_{x_2 \in \set{\pm 1}} \frac{\tp{1 + b x_2} \tp{p + r \tp{x_2 - b}}^2}{2 \tp{1 + a t x_2}} \\
        &\ge C_1 A \sum_{x_2 \in \set{\pm 1}} \frac{1 + b x_2}{2} \tp{p + r \tp{x_2 - b}}^2 \\
        &= C_1 \tp{A p^2 + A B r^2},
    \end{aligned}
    \]
    where \(C_1 > 0\) is a universal constant. Symmetrically,
    \[
    \norm{D_2 f}_{L^2\tp{\mu}}^2 \ge C_1 \tp{B q^2 + A B r^2}.
    \]
    Since \(0 < A, B \le 1\), we have
    \[
    \begin{aligned}
        Q &\ge - C_0 A B \abs{t}^3 \tp{\abs{p q} + \abs{p r} + \abs{q r}} \\
        &\ge - C_0 A B \abs{t}^3 \tp{p^2 + q^2 + r^2} \\
        &\ge - \frac{C_0}{C_1} \abs{t}^3 \tp{\norm{D_1 f}_{L^2\tp{\mu}}^2 + \norm{D_2 f}_{L^2\tp{\mu}}^2}.
    \end{aligned}
    \]
    Taking \(C \defeq C_0 / C_1\) completes the proof.
\end{proof}

\subsection{SK Estimates and \Poincare Inequality}
\label{subsec:sk-estimates-poincare}

We now verify the deterministic criterion from \Cref{thm:deterministic-poincare-criterion} for the SK interaction matrix. The required random-matrix estimates are collected in the following lemma.

\begin{lemma}
    For every fixed \(\beta > 0\), let \(\*J\) be the SK interaction matrix at inverse temperature \(\beta\), as defined in \eqref{eq:sk-matrix}, and set \(\*T \defeq \tp{\tanh J_{ij}}_{i, j = 1}^n\). There exists a deterministic positive sequence \(\zeta_n \to 0\), independent of \(\beta\), such that the following bounds hold simultaneously with probability at least \(1 - o_{\beta}\tp{1}\):
    \[
    \max_{i, j \in \stp{n}} \abs{T_{ij}} \le \zeta_n, \qquad \norm{\*T}_{\-{op}} \le 2 \beta + \zeta_n, \qquad \norm{\*T \circ \*T}_{\-{op}} \le \beta^2 + \zeta_n, \qquad r_3\tp{\*T} \le \zeta_n.
    \]
    \label{lem:sk-interaction-estimates}
\end{lemma}

\begin{proof}
    A Gaussian union bound and the inequality \(\abs{\tanh z} \le \abs{z}\) show that, with probability at least \(1 - o_{\beta}\tp{1}\),
    \begin{equation}
        \max_{i, j \in \stp{n}} \abs{T_{ij}} \le \max_{i, j \in \stp{n}} \abs{J_{ij}} = O\tp{\beta \sqrt{\frac{\log n}{n}}} = O\tp{\sqrt{\frac{\log n}{n}}}.
        \label{eq:max-J-high-probability-bound}
    \end{equation}
    On the same event, we have
    \begin{equation}
        r_3\tp{\*T} \le n \tp{\max_{i, j \in \stp{n}} \abs{J_{ij}}}^3 = O\tp{\sqrt{\frac{\log^3 n}{n}}}.
        \label{eq:r3-high-probability-bound}
    \end{equation}
    Let \(\tilde{\*J} \defeq \beta \*W\), where \(\*W \sim \-{GOE}\tp{n}\), so that \(\*J\) is obtained from \(\tilde{\*J}\) by setting its diagonal to zero. This modification changes the operator norm by at most
    \[
    \max_{i \in \stp{n}} \abs{\tilde{J}_{ii}} = O\tp{\beta \sqrt{\frac{\log n}{n}}} = o\tp{1}
    \]
    with probability at least \(1 - o_{\beta}\tp{1}\). Standard random matrix theory therefore yields
    \[
    \norm{\*J}_{\-{op}} = \beta \tp{2 + o\tp{1}} = 2 \beta + o\tp{1}
    \]
    with probability at least \(1 - o_{\beta}\tp{1}\). On the event in \eqref{eq:max-J-high-probability-bound}, the bound \(\abs{\tanh z - z} \le \abs{z}^3 / 3\) implies
    \[
    \norm{\*T - \*J}_{\-{op}} \le \norm{\*T - \*J}_{\-{F}} \le \frac{1}{3} \tp{\sum_{i, j = 1}^n J_{ij}^6}^{1 / 2} = O\tp{\sqrt{\frac{\log^3 n}{n}}}.
    \]
    Combining the preceding estimates, we conclude that, with probability at least \(1 - o_{\beta}\tp{1}\),
    \begin{equation}
        \norm{\*T}_{\-{op}} \le \norm{\*J}_{\-{op}} + \norm{\*T - \*J}_{\-{op}} = 2 \beta + o\tp{1}.
        \label{eq:sk-tanh-operator-norm-bound}
    \end{equation}
    Because \(\*T \circ \*T\) is symmetric and entrywise nonnegative, its operator norm is at most its maximum row sum. Using \(\abs{\tanh z} \le \abs{z}\), we obtain
    \begin{equation}
        \norm{\*T \circ \*T}_{\-{op}} \le \max_{i \in \stp{n}} \sum_{j = 1}^n T_{ij}^2 \le \max_{i \in \stp{n}} \sum_{j = 1}^n J_{ij}^2 = \beta^2 \tp{1 + o\tp{1}} = \beta^2 + o\tp{1}
        \label{eq:sk-tanh-squared-operator-norm-bound}
    \end{equation}
    with probability at least \(1 - o_{\beta}\tp{1}\), where the final estimate follows from chi-square concentration and a union bound over the rows.

    Choosing \(\zeta_n\) to be any deterministic positive sequence tending to zero that dominates the four error terms in \eqref{eq:max-J-high-probability-bound}, \eqref{eq:r3-high-probability-bound}, \eqref{eq:sk-tanh-operator-norm-bound}, \eqref{eq:sk-tanh-squared-operator-norm-bound} completes the proof.
\end{proof}

We can strengthen the preceding lemma to the following corollary, which will be applied in \Cref{subsec:optimal-mlsi-localization-schemes}.

\begin{corollary}
    Consider the same setting as in \Cref{lem:sk-interaction-estimates}. On the event constructed in its proof, the bounds in \Cref{lem:sk-interaction-estimates} can be strengthened to
    \[
    \max_{i, j \in \stp{n}} \abs{T_{ij}^{\tp{s}}} \le \zeta_n, \qquad \norm{\*T^{\tp{s}}}_{\-{op}} \le 2 \beta s + \zeta_n, \qquad \norm{\*T^{\tp{s}} \circ \*T^{\tp{s}}}_{\-{op}} \le s \beta^2 + \zeta_n, \qquad r_3\tp{\*T^{\tp{s}}} \le \zeta_n
    \]
    for all \(s \in \stp{0, 1}\), where \(\*T^{\tp{s}} \defeq \tp{\tanh\tp{s J_{ij}}}_{i, j = 1}^n\).
    \label{cor:sk-interaction-estimates-uniform}
\end{corollary}

\begin{proof}
    For \(s \in \stp{0, 1}\), the bounds \(\abs{\tanh\tp{s x}} \le s \abs{x}\) and \(\abs{\tanh\tp{s x} - s x} \le s^3 \abs{x}^3 / 3\) give
    \[
    \begin{aligned}
        \max_{i, j \in \stp{n}} \abs{T_{ij}^{\tp{s}}} &\le s \max_{i, j \in \stp{n}} \abs{J_{ij}}, \\
        \norm{\*T^{\tp{s}}}_{\-{op}} &\le s \norm{\*J}_{\-{op}} + \frac{s^3}{3} \tp{\sum_{i, j = 1}^n J_{ij}^6}^{1 / 2}, \\
        \norm{\*T^{\tp{s}} \circ \*T^{\tp{s}}}_{\-{op}} &\le s^2 \max_{i \in \stp{n}} \sum_{j = 1}^n J_{ij}^2, \\
        r_3\tp{\*T^{\tp{s}}} &\le s^3 n \tp{\max_{i, j \in \stp{n}} \abs{J_{ij}}}^3.
    \end{aligned}
    \]
    Since \(s \le 1\), the estimates from the proof of \Cref{lem:sk-interaction-estimates} hold uniformly over \(s \in \stp{0, 1}\) on the same event and with the same error sequence, proving the claim.
\end{proof}

We now apply the preceding estimates to prove the optimal-order \Poincare inequality for the SK model.

\begin{proof}[Proof of \Cref{thm:sk-optimal-poincare}]
    Let \(\zeta_n\) be the sequence from \Cref{lem:sk-interaction-estimates}, and define \(\eta_n \defeq \tp{C + 2}\zeta_n\), where \(C\) is the universal constant in \Cref{thm:deterministic-poincare-criterion}. In particular, \(\eta_n\) is independent of \(\beta\). Since \(\beta^2 \le \beta\), the bounds in \Cref{lem:sk-interaction-estimates} imply that, on its high-probability event,
    \[
    \max\set{\norm{\*T}_{\-{op}}, 2 \norm{\*T \circ \*T}_{\-{op}}} + C r_3\tp{\*T} \le \max\set{2 \beta + \zeta_n, 2 \beta^2 + 2 \zeta_n} + C \zeta_n \le 2 \beta + \eta_n.
    \]
    Moreover, since \(\zeta_n \to 0\), for all sufficiently large \(n\) we have \(\max_{i, j \in \stp{n}} \abs{T_{ij}} \le \zeta_n \le \delta\), where \(\delta\) is the constant in \Cref{thm:deterministic-poincare-criterion}. Applying that theorem therefore gives, simultaneously for every \(\*h \in \bb{R}^n\),
    \[
    \frac{1 - 2 \beta - \eta_n}{n} \*{Var}_{\mu_{\*J, \*h}}\stp{f} \le \+E_{P_{\*J, \*h}}\tp{f, f}, \qquad \forall f \in L^2\tp{\mu_{\*J, \*h}}.
    \]
    This is the claimed \Poincare inequality.
\end{proof}

\section{Optimal-Order Modified Log-Sobolev Inequality}
\label{sec:optimal-mlsi-mixing}

Building on the \Poincare inequality from \Cref{sec:optimal-poincare}, we prove an optimal-order modified log-Sobolev inequality for the SK model and thereby obtain an \(O_{\beta}\tp{n \log n}\) mixing time. We establish the modified log-Sobolev inequality in \Cref{subsec:optimal-mlsi-localization-schemes} and deduce the mixing-time bound in \Cref{subsec:optimal-mixing}.

\subsection{Optimal Modified Log-Sobolev Inequality via Localization Schemes}
\label{subsec:optimal-mlsi-localization-schemes}

To upgrade the optimal-order \Poincare inequality to an optimal-order modified log-Sobolev inequality, we use the localization-scheme framework of \cite{CE25}. A \Poincare inequality uniform over external fields yields uniform covariance bounds by testing on linear functions, which can then be upgraded to entropy inequalities via the entropic-stability approach of Chen and Eldan.\footnote{An alternative covariance-to-log-Sobolev route is the multiscale Bakry--\Emery/Polchinski approach; see \cite[Section~3.7]{BBD24} for a comparison of these two approaches.} We recall only the consequence needed here.

\begin{theorem}[{Consequence of \cite[Theorem~49]{CE25}}]
    Let \(\mu_{\*J, \*h}\) be an Ising model on \(\set{\pm 1}^n\), where \(\*J\) is symmetric and has zero diagonal. Suppose that for each \(s \in \stp{0, 1}\),
    \[
    \norm{\cov{\mu_{s \*J, \*h'}}}_{\-{op}} \le \alpha\tp{s}, \quad \forall \*h' \in \bb{R}^n,
    \]
    and define
    \[
    \rho \defeq \frac{1}{n} \exp\tp{-\tp{\lambda_{\max}\tp{\*J} - \lambda_{\min}\tp{\*J}} \int_0^1 \alpha\tp{s} \dd s}.
    \]
    Then the Glauber dynamics kernel \(P_{\*J, \*h}\) satisfies the modified log-Sobolev inequality with constant \(\rho\):
    \[
    \rho \*{Ent}_{\mu_{\*J, \*h}}\stp{f} \le \+E_{P_{\*J, \*h}}\tp{\log f, f}, \quad \forall f \in L^2\tp{\mu_{\*J, \*h}} \cmid f \ge 0.
    \]
    \label{thm:stochastic-localization-mlsi}
\end{theorem}

The required covariance bounds follow directly from the \Poincare inequality.

\begin{lemma}
    Let \(\mu\) be a probability measure on \(\set{\pm 1}^n\) and let \(P\) be the Glauber dynamics kernel for \(\mu\). Suppose that \(P\) satisfies the \Poincare inequality with constant \(\gamma\):
    \[
    \gamma \*{Var}_{\mu}\stp{f} \le \+E_P\tp{f, f}, \quad \forall f \in L^2\tp{\mu}.
    \]
    Then
    \[
    \norm{\cov{\mu}}_{\-{op}} \le \frac{1}{\gamma n}.
    \]
    \label{lem:poincare-covariance-bound}
\end{lemma}

\begin{proof}[Proof of \Cref{lem:poincare-covariance-bound}]
    Fix \(\*u \in \bb{R}^n\), and set \(\ell_{\*u}\tp{\*x} \defeq \*u^\top \*x\). Then
    \[
    \*{Var}_{\mu}\stp{\ell_{\*u}} = \*u^\top \cov{\mu} \*u.
    \]
    Moreover, \Cref{prop:integrated-carre-du-champ}, \Cref{lem:local-variance-difference-identity}, \(d_i \ell_{\*u} = u_i\) and \(v_i \le 1\) give
    \[
    \+E_P\tp{\ell_{\*u}, \ell_{\*u}} = \frac{1}{n} \sum_{i = 1}^n \norm{D_i \ell_{\*u}}_{L^2\tp{\mu}}^2 = \frac{1}{n} \sum_{i = 1}^n \*E_{\mu}\stp{v_i \tp{d_i \ell_{\*u}}^2} = \frac{1}{n} \sum_{i = 1}^n \*E_{\mu}\stp{v_i} u_i^2 \le \frac{1}{n} \norm{\*u}_2^2.
    \]
    Applying the assumed \Poincare inequality to \(\ell_{\*u}\) yields \(\*u^\top \cov{\mu} \*u \le \norm{\*u}_2^2 / \tp{\gamma n}\). Taking the supremum over \(\norm{\*u}_2 = 1\) proves the claim.
\end{proof}

We now combine this covariance estimate with \Cref{thm:stochastic-localization-mlsi} to prove an optimal-order modified log-Sobolev inequality for the SK model.

\begin{theorem}
    For every fixed \(\beta < 1 / 2\), let \(\*J\) be the SK interaction matrix at inverse temperature \(\beta\), as defined in \eqref{eq:sk-matrix}. For all sufficiently large \(n \ge n_0\tp{\beta}\), with probability at least \(1 - o_{\beta}\tp{1}\), the Glauber dynamics kernel \(P_{\*J, \*h}\) satisfies the following modified log-Sobolev inequality simultaneously over all external fields \(\*h \in \bb{R}^n\):
    \[
    \frac{\tp{1 - 2 \beta}^2 - o_{\beta}\tp{1}}{n} \*{Ent}_{\mu_{\*J, \*h}}\stp{f} \le \+E_{P_{\*J, \*h}}\tp{\log f, f}, \quad \forall f \in L^2\tp{\mu_{\*J, \*h}} \cmid f \ge 0.
    \]
    \label{thm:sk-optimal-mlsi}
\end{theorem}

\begin{proof}[Proof of \Cref{thm:sk-optimal-mlsi}]
    Work on the high-probability event in \Cref{cor:sk-interaction-estimates-uniform}, and let \(C\) be the universal constant in \Cref{thm:deterministic-poincare-criterion}. The bounds in the proof of \Cref{cor:sk-interaction-estimates-uniform} imply that uniformly over \(s \in \stp{0, 1}\),
    \[
    \max\set{\norm{\*T^{\tp{s}}}_{\-{op}}, \, 2 \norm{\*T^{\tp{s}} \circ \*T^{\tp{s}}}_{\-{op}}} + C r_3\tp{\*T^{\tp{s}}} \le 2 \beta s + \eta_n,
    \]
    where \(\eta_n \defeq \tp{C + 2}\zeta_n \to 0\). For all sufficiently large \(n\), \Cref{thm:deterministic-poincare-criterion} therefore gives, simultaneously for every \(s \in \stp{0, 1}\) and \(\*h' \in \bb{R}^n\), a \Poincare inequality for \(P_{s \*J, \*h'}\) with constant
    \[
    \gamma_n\tp{s} \defeq \frac{1 - 2 \beta s - \eta_n}{n} > 0.
    \]
    By \Cref{lem:poincare-covariance-bound},
    \[
    \norm{\cov{\mu_{s \*J, \*h'}}}_{\-{op}} \le \alpha_n\tp{s} \defeq \frac{1}{1 - 2 \beta s - \eta_n}, \quad \forall s \in \stp{0, 1}, \, \forall \*h' \in \bb{R}^n.
    \]
    On the same event, \(\norm{\*J}_{\-{op}} \le 2 \beta + \zeta_n\), and hence
    \[
    \lambda_{\max}\tp{\*J} - \lambda_{\min}\tp{\*J} \le 4 \beta + 2 \zeta_n.
    \]
    Applying \Cref{thm:stochastic-localization-mlsi} with \(\alpha_n\) and using
    \[
    \int_0^1 \alpha_n\tp{s} \dd s = \frac{1}{2 \beta} \log\tp{\frac{1 - \eta_n}{1 - 2 \beta - \eta_n}}
    \]
    shows that the modified log-Sobolev constant for \(P_{\*J, \*h}\) is at least
    \[
    \frac{1}{n} \exp\tp{-\tp{4 \beta + 2 \zeta_n} \frac{1}{2 \beta} \log\tp{\frac{1 - \eta_n}{1 - 2 \beta - \eta_n}}} = \frac{\tp{1 - 2 \beta}^2 - o_{\beta}\tp{1}}{n}
    \]
    as claimed.
\end{proof}

\subsection{Optimal Mixing of Glauber Dynamics}
\label{subsec:optimal-mixing}

Although \Cref{thm:functional-inequality-mixing,thm:sk-optimal-mlsi} do not directly yield a mixing bound uniform in the external field, a preliminary coupon-collector phase provides a warm start and leads to an \(O_{\beta}\tp{n \log n}\) mixing time.

\begin{lemma}
    Let \(P_{\*J, \*h}\) be the Glauber dynamics kernel for the Ising model \(\mu_{\*J, \*h}\) on \(\set{\pm 1}^n\), where \(\*J\) is symmetric and has zero diagonal. Fix an initial configuration and an integer \(N \ge n\). Let \(\nu\) be the law after \(N\) steps, conditional on every site having been updated at least once. Then
    \[
    \frac{\dd \nu}{\dd \mu_{\*J, \*h}} \le \exp\tp{4 \sum_{i, j = 1}^n \abs{J_{ij}}}.
    \]
    \label{lem:coupon-collector-warmness}
\end{lemma}

\begin{proof}[Proof of \Cref{lem:coupon-collector-warmness}]
    Let \(\*X \sim \mu_{\*J, \*h}\). For \(i \in \stp{n}\), \(x_i \in \set{\pm 1}\), and \(\*x_{-i} \in \set{\pm 1}^{n - 1}\), set
    \[
    p_i\tp{x_i \mid \*x_{-i}} \defeq \*{Pr}\stp{X_i = x_i \mid \*X_{-i} = \*x_{-i}}.
    \]
    Fix a deterministic site-covering schedule of \(N\) steps, and let \(q\) be the law of the final configuration. Order the sites by their last-update times and expose those terminal update outcomes chronologically. Conditional on the previously exposed outcomes, the probability of prescribing the final value \(x_i\) is at most \(\max_{\*z} p_i\tp{x_i \mid \*z}\). Multiplying these bounds yields the next display.
    \[
    q\tp{\*x} \le \prod_{i = 1}^n \max_{\*z \in \set{\pm 1}^{n - 1}} p_i\tp{x_i \mid \*z}.
    \]
    Another chain-rule expansion of \(\mu_{\*J, \*h}\) in any coordinate order gives
    \[
    \mu_{\*J, \*h}\tp{\*x} \ge \prod_{i = 1}^n \min_{\*z \in \set{\pm 1}^{n - 1}} p_i\tp{x_i \mid \*z}.
    \]
    Indeed, each partial conditional probability in this chain-rule expansion is a convex combination of the full single-site conditionals \(p_i\tp{x_i \mid \*z}\), and hence is bounded below by their minimum. Since
    \[
    \log p_i\tp{x_i \mid \*z} = x_i h_{-i}\tp{\*z} - \log\tp{2 \cosh h_{-i}\tp{\*z}}, \qquad h_{-i}\tp{\*z} \defeq h_i + \sum_{j \ne i} J_{ij} z_j
    \]
    and the right-hand side is \(2\)-Lipschitz in \(h_{-i}\tp{\*z}\), whose range has diameter at most \(2 \sum_{j \ne i} \abs{J_{ij}}\), we have
    \[
    \frac{\max_{\*z \in \set{\pm 1}^{n - 1}} p_i\tp{x_i \mid \*z}}{\min_{\*z \in \set{\pm 1}^{n - 1}} p_i\tp{x_i \mid \*z}} \le \exp\tp{4 \sum_{j \ne i} \abs{J_{ij}}}.
    \]
    Therefore,
    \[
    \frac{q\tp{\*x}}{\mu_{\*J, \*h}\tp{\*x}} \le \exp\tp{4 \sum_{i = 1}^n \sum_{j \ne i} \abs{J_{ij}}} = \exp\tp{4 \sum_{i, j = 1}^n \abs{J_{ij}}}, \quad \forall \*x \in \set{\pm 1}^n.
    \]
    Conditional on covering every site, the random update schedule is a mixture of deterministic covering schedules, so the same bound holds for \(\nu\).
\end{proof}

We now combine the warm start from \Cref{lem:coupon-collector-warmness} with the optimal-order modified log-Sobolev inequality to prove the optimal \(O_{\beta}\tp{n \log n}\) mixing time.

\begin{proof}[Proof of \Cref{thm:sk-optimal-mixing}]
    Work on the intersection of the high-probability event in \Cref{thm:sk-optimal-mlsi} with the event in \eqref{eq:max-J-high-probability-bound}. For all sufficiently large \(n \ge n_0\tp{\beta}\), there exists \(c_{\beta} > 0\) such that \(P_{\*J, \*h}\) has modified log-Sobolev constant \(\rho \ge c_{\beta} / n\), simultaneously for every \(\*h \in \bb{R}^n\). Moreover, with \(S_{\*J} \defeq \sum_{i, j = 1}^n \abs{J_{ij}}\),
    \[
    S_{\*J} \le n^2 \max_{i, j \in \stp{n}} \abs{J_{ij}} = O\tp{n^{3 / 2} \sqrt{\log n}}.
    \]
    Fix \(\*h \in \bb{R}^n\), \(\*x \in \set{\pm 1}^n\), and \(\eps \in \tp{0, 1}\), and set \(N_0 \defeq \ceil{n \log\tp{2 n / \eps}}\). If \(\+A\) is the event that every site is updated during the first \(N_0\) steps, then
    \[
    \*{Pr}\stp{\+A^c} \le n \tp{1 - \frac{1}{n}}^{N_0} \le \frac{\eps}{2}.
    \]
    Let \(\nu\) be the law at time \(N_0\) conditional on \(\+A\). By \Cref{lem:coupon-collector-warmness}, \(\nu\) is \(M_{\*J}\)-warm for \(M_{\*J} \defeq \exp\tp{1 + 4 S_{\*J}}\). Set
    \[
    N_1 \defeq \ceil{\frac{1}{\rho} \tp{\log\tp{\frac{2}{\eps^2}} + \log\tp{1 + 4 S_{\*J}}}} = O_{\beta}\tp{n \log\tp{\frac{n}{\eps}}}.
    \]
    By \Cref{thm:functional-inequality-mixing}, \(d_{\-{TV}}\tp{\nu P_{\*J, \*h}^{N_1}, \mu_{\*J, \*h}} \le \eps / 2\). Decomposing the law at time \(N_0\) according to \(\+A\), we obtain
    \[
    d_{\-{TV}}\tp{P_{\*J, \*h}^{N_0 + N_1}\tp{\*x, \cdot}, \mu_{\*J, \*h}} \le \*{Pr}\stp{\+A^c} + d_{\-{TV}}\tp{\nu P_{\*J, \*h}^{N_1}, \mu_{\*J, \*h}} \le \eps.
    \]
    Since \(N_0 + N_1 = O_{\beta}\tp{n \log\tp{n / \eps}}\) uniformly over \(\*x\) and \(\*h\), the claim follows.
\end{proof}

\section*{Statement on AI Use}

The proof of the \Poincare inequality in \Cref{sec:optimal-poincare} was initially generated by GPT-5.6 Sol Ultra over two rounds of interaction. The author independently verified and streamlined the argument, extended it to establish a modified log-Sobolev inequality and an \(O_{\beta}\tp{n \log n}\) mixing time, and composed the manuscript.

GPT-5.6 Sol was also used for editorial assistance, including improving the exposition and identifying typographical errors. The author takes full responsibility for the contents of this paper.

\section*{Acknowledgments}

The author is especially grateful to Roland Bauerschmidt for pointing out relevant literature and for valuable comments on earlier work. The author also thanks Zongchen Chen, Heng Guo, and Xinyuan Zhang for helpful discussions.

\bibliographystyle{alpha}
\bibliography{refs}

@inproceedings{AKV24,
  author    = {Anari, Nima and Koehler, Frederic and Vuong, Thuy-Duong},
  title     = {Trickle-down in localization schemes and applications},
  booktitle = {Proceedings of the 56th Annual ACM Symposium on Theory of Computing},
  pages     = {1094--1105},
  year      = {2024}
}

@article{CE25,
  author  = {Chen, Yuansi and Eldan, Ronen},
  title   = {Localization schemes: a framework for proving mixing bounds for {Markov} chains},
  journal = {Duke Mathematical Journal},
  volume  = {174},
  number  = {8},
  pages   = {1431--1510},
  year    = {2025}
}

@article{SK75,
  author  = {Sherrington, David and Kirkpatrick, Scott},
  title   = {Solvable model of a spin-glass},
  journal = {Physical Review Letters},
  volume  = {35},
  number  = {26},
  pages   = {1792--1796},
  year    = {1975}
}

@article{BKMR25,
  author  = {Bandeira, Afonso S. and Kireeva, Anastasia and Maillard, Antoine and R{\"o}dder, Almut},
  title   = {{Randomstrasse101}: open problems of 2024},
  journal = {arXiv preprint arXiv:2504.20539},
  year    = {2025}
}

@article{EKZ22,
  author  = {Eldan, Ronen and Koehler, Frederic and Zeitouni, Ofer},
  title   = {A spectral condition for spectral gap: fast mixing in high-temperature {Ising} models},
  journal = {Probability Theory and Related Fields},
  volume  = {182},
  number  = {3--4},
  pages   = {1035--1051},
  year    = {2022}
}

@inproceedings{AKJVP22,
  author    = {Anari, Nima and Jain, Vishesh and Koehler, Frederic and Pham, Huy Tuan and Vuong, Thuy-Duong},
  title     = {Entropic independence: optimal mixing of down-up random walks},
  booktitle = {Proceedings of the 54th Annual ACM SIGACT Symposium on Theory of Computing},
  pages     = {1418--1430},
  year      = {2022}
}

@article{DLSS26b,
  author  = {Davies, Ewan and Lee, Holden and Sandhu, Juspreet Singh and Shi, Jonathan},
  title   = {Weak {Poincar\'e} inequalities via approximate stochastic localization: application to sampling the {Sherrington--Kirkpatrick} model},
  journal = {arXiv preprint arXiv:2607.08160},
  year    = {2026}
}

@article{DLSS26a,
  author  = {Davies, Ewan and Lee, Holden and Sandhu, Juspreet Singh and Shi, Jonathan},
  title   = {Potential Hessian ascent {III}: sampling the {Sherrington--Kirkpatrick} model at $\beta < 1/2$},
  journal = {arXiv preprint arXiv:2605.03718},
  year    = {2026}
}

@inproceedings{EAMS22,
  author       = {El Alaoui, Ahmed and Montanari, Andrea and Sellke, Mark},
  title        = {Sampling from the {Sherrington--Kirkpatrick} {Gibbs} measure via algorithmic stochastic localization},
  booktitle    = {2022 IEEE 63rd Annual Symposium on Foundations of Computer Science (FOCS)},
  pages        = {323--334},
  year         = {2022},
  organization = {IEEE}
}

@article{Cel24,
  author  = {Celentano, Michael},
  title   = {{Sudakov--Fernique} post-{AMP}, and a new proof of the local convexity of the {TAP} free energy},
  journal = {The Annals of Probability},
  volume  = {52},
  number  = {3},
  pages   = {923--954},
  year    = {2024}
}

@book{LPW17,
  author    = {Levin, David A. and Peres, Yuval and Wilmer, Elizabeth L.},
  title     = {{Markov} chains and mixing times},
  edition   = {2nd},
  year      = {2017},
  publisher = {American Mathematical Society},
  address   = {Providence, Rhode Island},
  isbn      = {978-1-4704-2962-1}
}

@inproceedings{BT03,
  author    = {Bobkov, Sergey G. and Tetali, Prasad},
  title     = {Modified log-{Sobolev} inequalities, mixing and hypercontractivity},
  booktitle = {Proceedings of the Thirty-Fifth Annual ACM Symposium on Theory of Computing},
  pages     = {287--296},
  year      = {2003}
}

@article{KKO13,
  author  = {Kondratiev, Yuri and Kuna, Tobias and Ohlerich, Nataliya},
  title   = {Spectral gap for {Glauber} type dynamics for a special class of potentials},
  journal = {Electronic Journal of Probability},
  volume  = {18},
  number  = {42},
  pages   = {1--18},
  year    = {2013}
}

@article{Joh17,
  author  = {Johnson, Oliver},
  title   = {A discrete log-{Sobolev} inequality under a {Bakry-\'Emery} type condition},
  journal = {Annales de l'Institut Henri Poincar{\'e}, Probabilit{\'e}s et Statistiques},
  volume  = {53},
  number  = {4},
  pages   = {1952--1970},
  year    = {2017}
}

@article{GJMPPS26,
  author  = {G{\"o}bel, Andreas and Jenssen, Matthew and Michelen, Marcus and Pappik, Marcus and Perkins, Will and Schiller, Leon},
  title   = {A simple proof of rapid mixing on random regular graphs beyond uniqueness},
  journal = {arXiv preprint arXiv:2606.27545},
  year    = {2026}
}

@article{BCDPP06,
  author  = {Boudou, Anne-Severine and Caputo, Pietro and Dai Pra, Paolo and Posta, Gustavo},
  title   = {Spectral gap estimates for interacting particle systems via a {Bochner}-type identity},
  journal = {Journal of Functional Analysis},
  volume  = {232},
  number  = {1},
  pages   = {222--258},
  year    = {2006}
}

@book{BGL14,
  author    = {Bakry, Dominique and Gentil, Ivan and Ledoux, Michel},
  title     = {Analysis and geometry of {Markov} diffusion operators},
  series    = {Grundlehren der mathematischen Wissenschaften},
  volume    = {348},
  year      = {2014},
  publisher = {Springer}
}

@article{Eld13,
  author  = {Eldan, Ronen},
  title   = {Thin shell implies spectral gap up to polylog via a stochastic localization scheme},
  journal = {Geometric and Functional Analysis},
  volume  = {23},
  number  = {2},
  pages   = {532--569},
  year    = {2013}
}

@inproceedings{CCCYZ25,
  author       = {Chen, Xiaoyu and Chen, Zejia and Chen, Zongchen and Yin, Yitong and Zhang, Xinyuan},
  title        = {Rapid mixing on random regular graphs beyond uniqueness},
  booktitle    = {2025 IEEE 66th Annual Symposium on Foundations of Computer Science (FOCS)},
  pages        = {2170--2193},
  year         = {2025},
  organization = {IEEE}
}

@article{Sel25,
  author  = {Sellke, Mark},
  title   = {Exponentially slow mixing of the low temperature {SK} model},
  journal = {arXiv preprint arXiv:2511.22621},
  year    = {2025}
}

@article{SZ81,
  author  = {Sompolinsky, Haim and Zippelius, Annette},
  title   = {Dynamic theory of the spin-glass phase},
  journal = {Physical Review Letters},
  volume  = {47},
  number  = {5},
  pages   = {359--362},
  year    = {1981}
}

@book{MPV87,
  author    = {M{\'e}zard, Marc and Parisi, Giorgio and Virasoro, Miguel Angel},
  title     = {Spin glass theory and beyond: an introduction to the replica method and its applications},
  series    = {World Scientific Lecture Notes in Physics},
  volume    = {9},
  year      = {1987},
  publisher = {World Scientific Publishing Company}
}

@incollection{BE85,
  author    = {Bakry, Dominique and {\'E}mery, Michel},
  title     = {Diffusions hypercontractives},
  booktitle = {S{\'e}minaire de Probabilit{\'e}s XIX, 1983/84},
  editor    = {Az{\'e}ma, Jacques and Yor, Marc},
  series    = {Lecture Notes in Mathematics},
  volume    = {1123},
  pages     = {177--206},
  year      = {1985},
  publisher = {Springer},
  address   = {Berlin}
}

@misc{Che26,
  author = {Chewi, Sinho},
  title  = {Log-concave sampling},
  year   = {2026},
  note   = {\url{https://chewisinho.github.io/main.pdf}. Version: 2026}
}

@article{BB19,
  author  = {Bauerschmidt, Roland and Bodineau, Thierry},
  title   = {A very simple proof of the {LSI} for high temperature spin systems},
  journal = {Journal of Functional Analysis},
  volume  = {276},
  number  = {8},
  pages   = {2582--2588},
  year    = {2019}
}

@article{BBD24,
  author  = {Bauerschmidt, Roland and Bodineau, Thierry and Dagallier, Benoit},
  title   = {Stochastic dynamics and the {Polchinski} equation: an introduction},
  journal = {Probability Surveys},
  volume  = {21},
  pages   = {200--290},
  year    = {2024}
}

@article{BEAR26,
  author  = {Bandeira, Afonso S. and El Alaoui, Ahmed and R{\"o}dder, Almut},
  title   = {Mixing of {Glauber} dynamics on high overlap {Gibbs} measures},
  journal = {arXiv preprint arXiv:2607.06813},
  year    = {2026}
}

@article{HMP24,
  author  = {Huang, Brice and Montanari, Andrea and Pham, Huy Tuan},
  title   = {Sampling from spherical spin glasses in total variation via algorithmic stochastic localization},
  journal = {arXiv preprint arXiv:2404.15651},
  year    = {2024}
}

@inproceedings{AJKVP24,
  author       = {Anari, Nima and Jain, Vishesh and Koehler, Frederic and Pham, Huy Tuan and Vuong, Thuy-Duong},
  title        = {Universality of spectral independence with applications to fast mixing in spin glasses},
  booktitle    = {Proceedings of the 2024 Annual ACM-SIAM Symposium on Discrete Algorithms (SODA)},
  pages        = {5029--5056},
  year         = {2024},
  organization = {SIAM}
}

@article{ABXY24,
  author  = {Adhikari, Arka and Brennecke, Christian and Xu, Changji and Yau, Horng-Tzer},
  title   = {Spectral gap estimates for mixed $p$-spin models at high temperature},
  journal = {Probability Theory and Related Fields},
  volume  = {189},
  number  = {3--4},
  pages   = {879--907},
  year    = {2024}
}

@article{GJ19,
  author  = {Gheissari, Reza and Jagannath, Aukosh},
  title   = {On the spectral gap of spherical spin glass dynamics},
  journal = {Annales de l'Institut Henri Poincar{\'e}, Probabilit{\'e}s et Statistiques},
  volume  = {55},
  number  = {2},
  pages   = {756--776},
  year    = {2019}
}

@inproceedings{HMRW25,
  author    = {Huang, Brice and Mohanty, Sidhanth and Rajaraman, Amit and Wu, David X.},
  title     = {Weak {Poincar\'e} inequalities, simulated annealing, and sampling from spherical spin glasses},
  booktitle = {Proceedings of the 57th Annual ACM Symposium on Theory of Computing},
  pages     = {915--923},
  year      = {2025}
}

@article{EAMS25,
  author  = {El Alaoui, Ahmed and Montanari, Andrea and Sellke, Mark},
  title   = {Sampling from mean-field {Gibbs} measures via diffusion processes},
  journal = {Probability and Mathematical Physics},
  volume  = {6},
  number  = {3},
  pages   = {961--1022},
  year    = {2025}
}

@article{GT02,
  author  = {Guerra, Francesco and Toninelli, Fabio L.},
  title   = {Central limit theorem for fluctuations in the high temperature region of the {Sherrington--Kirkpatrick} spin glass model},
  journal = {Journal of Mathematical Physics},
  volume  = {43},
  number  = {12},
  pages   = {6224--6237},
  year    = {2002}
}

@article{Ton02,
  author  = {Toninelli, Fabio Lucio},
  title   = {About the {Almeida--Thouless} transition line in the {Sherrington--Kirkpatrick} mean-field spin glass model},
  journal = {Europhysics Letters},
  volume  = {60},
  number  = {5},
  pages   = {764--767},
  year    = {2002}
}

\end{document}